\documentclass[11pt]{amsart}

\usepackage{amssymb}
\usepackage{latexsym,bm}

\numberwithin{equation}{section} \theoremstyle{section}
\newtheorem{Remark}[equation]{Remark}
\theoremstyle{plain}
\newtheorem{De}[equation]{Definition}

\newtheorem{Theorem}[equation]{Theorem}

\newtheorem{Lemma}[equation]{Lemma}
\newtheorem{Cor}[equation]{Corollary}

\title{THE ROKHLIN PROPERTY FOR AUTOMORPHISMS
ON SIMPLE $C$*-ALGEBRAS}

\author{JIAJIE HUA}

\date{}

\begin{document}

\maketitle \markboth{JIAJIE HUA}{THE ROKHLIN PROPERTY FOR
AUTOMORPHISMS ON SIMPLE $C$*-ALGEBRAS}
\renewcommand{\thefootnote}{{}}
\footnote{\hspace{-14pt} {\it 2000 Mathematics Subject Classification}. Primary 46L55: Secondary 46L35, 46L40.The author was supported the National Natural Science
Foundation of China (Nos. 10771069, 10671068,10771161).}

\begin{abstract}  Let $\mathcal{A}$ be the class of unital separable simple amenable $C$*-algebras $A$ which
satisfy the Universal Coefficient Theorem for which $A\otimes
M_{\texttt{P}}$ has tracial rank zero for some supernatural number
$\texttt{p}$ of infinite type. Let $A\in \mathcal{A}$ and let
$\alpha$ be an automorphism of $A.$ Suppose that $\alpha$ has the tracial
Rokhlin property. Suppose also that there is an integer $J\geq 1$
such that $[\alpha^J]=[\mbox{id}_A]$ in $KL(A,A)$, we show that
$A\rtimes_{\alpha}\mathbb{Z}\in \mathcal{A}.$
\end{abstract}

\section{Introduction}

Alan Connes introduced the Rokhlin property in ergodic theory to
operator algebras (\cite{A.Connes1}). Several versions of the
Rokhlin property for automorphism on $C$*-algebra have been studied
(for example \cite{A.Kishimoto4},\cite{M.R},\cite{N. C.
Phillips3},\cite{R.Herman}). Given a unital $C$*-algebra $A$ and an
automorphism $\alpha$ of $A$, one may view the pair $(A,\alpha)$ as
a non-commutative dynamical system. To study its dynamical
structure, it is natural to introduce the notion of Rokhlin
property. Let $A$ be a unital simple A$\mathbb{T}$-algebra (direct limits of circle algebras) with real rank zero and let $\alpha$ be an automorphism of $A$. In several
cases, Kishimoto showed that if $\alpha$ is approximately inner and
has a Rokhlin property, then the crossed product
$A\rtimes_{\alpha}\mathbb{Z}$ is again a unital simple A$\mathbb{T}$
with real rank zero (\cite{A.Kishimoto3},\cite{A.Kishimoto4}).

In \cite{H.Lin3}, H. Lin proved the following result:  Let $A$ be a unital separable
simple $C$*-algebra with tracial rank zero and let $\alpha$ be an
automorphism of $A$. Suppose that $\alpha$ has a tracial Rokhlin property. Suppose also that there is an integer $J\geq1$
such that $[\alpha^J]=[\mbox{id}_A]$ in $KL(A,A)$. Then
$A\rtimes_{\alpha}\mathbb{Z}$ has tracial rank zero. Thus H. Lin generalizes the Kishimoto'results on unital simple A$\mathbb{T}$-algebras with real rank zero to $C$*-algebras with tracial rank zero.

In this paper, we generalize Lin's result as the following: Let $\mathcal{A}$ be the
class of unital separable simple amenable $C$*-algebras $A$ which
satisfy the Universal Coefficient Theorem for which $A\otimes
M_{\texttt{P}}$ has tracial rank zero for some supernatural number
$\texttt{p}$ of infinite type. Let $A\in \mathcal{A}$ and let
$\alpha$ be an automorphism of $A.$ Suppose that $\alpha$ has the tracial
Rokhlin property. Suppose also that there is an integer $J\geq 1$
such that $[\alpha^J]=[\mbox{id}_A]$ in $KL(A,A)$. Then
$A\rtimes_{\alpha}\mathbb{Z}\in \mathcal{A}.$

This paper is organized as follows. In Section 2 we introduce
notation on $C$*-algebras. In Section 3, we prove that,
under our hypotheses, $A\rtimes_{\alpha}\mathbb{Z}\in \mathcal{A}$.

\section{Notation}

We will use the following convention:

(1) Let $A$ be a $C$*-algebra, let $a\in A$ be a positive element
and let $p\in A$ be a projection. We write $[p]\leq [a]$ if there is
a projection $q\in \overline{aAa}$ and a partial isometry $v\in A$
such that $v^{*}v=p$ and $vv^{*}=q.$

(2) Let $A$ be a $C$*-algebra. We denote by Aut$(A)$ the
automorphism group of $A$. If $A$ is unital and $u\in A$ is a
unitary , we denote by ad$u$ the inner automorphism defined by
ad$u(a)=u^{*}au$ for all $a\in A.$

(3) Let $x\in A$, $\varepsilon>0$ and $\mathcal{F}\subset A.$ We
write $x\in_{\varepsilon}\mathcal{F},$ if
dist$(x,\mathcal{F})<\varepsilon$, or there is $y\in \mathcal{F}$
such that $\| x-y \|< \varepsilon .$

(4) Let $A$ be a unital $C$*-algebra and $T(A)$ the compact convex
set of tracial states of $A$.

(5) A unital $C$*-algebra is said to have real rank zero, written
RR$(A)=0$, if the set of invertible self-adjoint elements is dense
in self-adjoint elements of $A.$

(6) we say the order on projection over a unital $C$*-algebra $A$ is
determined by traces, if for any two projections $p,q\in A,
\tau(p)<\tau(q)$ for all $\tau\in T(A)$ implies that $p$ is
equivalent to a projection $p'\leq q.$

\begin{De}\cite{H.Lin1} Let $A$ be a unital simple $C$*-algebra.
Then $A$ is said to have tracial (topological) rank zero if for every $\varepsilon>0$, every
finite set $\mathcal{F}\subset A,$  and every
nonzero positive element $a\in A,$ there exists a finite dimensional
$C$*-subalgebra $B\subset A$ with id$_{B}=p$ such
that:\\
(1) $\| px-xp \|< \varepsilon$ for all $x\in \mathcal{F}$.\\
(2) $pap\in_{\varepsilon} B$ for all $x\in \mathcal{F}$.\\
(3) $[1-p]\leq [a].$
\end{De}

If $A$ has tracial rank zero, we write $\mathrm{TR}(A)=0.$

\begin{De} (The Jiang-Su algebra \cite{X. Jiang}) Denote by
$\mathcal{Z}$ the Jiang-Su algebra of unital infinite dimensional
simple $C$*-algebra which is an inductive limit of prime dimension
drop algebras with a unique tracial state, $(K_0(\mathcal{Z}),
K_0(\mathcal{Z})_+,[1_{\mathcal{Z}}])=(\mathbb{Z},\mathbb{N},1)$ and
$K_1(\mathcal{Z})=0.$
\end{De}

\begin{De} (A classifiable class of unital separable simple amenable
$C$*-algebra) Denote by $\mathcal{N}$ the class of all unital
separable amenable $C$*-algebras which satisfy the Universal
Coefficient Theorem.

For a supernatural number $\texttt{p},$ denote by $M_{\texttt{p}}$
the UHF algebra associated with $\texttt{p}$(see \cite{J.Dixmier}).

Let $\mathcal{A}$ denote the class of all unital separable simple
amenable $C$*-algebras $A$ in $\mathcal{N}$ for which $A\otimes
M_{\texttt{P}}$ has tracial rank zero for some supernatural number
$\texttt{p}$ of infinite type.
\end{De}

\begin{De} Let $A$ be a unital stably finite separable simple
amenable $C$*-algebra.

By Ell($A$) we mean the following:
$$(K_0(A),K_0(A)_+,[1_A],K_1(A),T(A),r_{A}),$$
where $r_{A}:T(A)\rightarrow S_{[1_A]}(K_0(A))$ is a surjective
continuous affine map such that $r_A(\tau)([p])=\tau(p)$ for all
projections $p\in A\otimes M_k$ for $k=1,2,\dots$.

Suppose that $B$ is another stably finite unital separable simple
$C$*-algebra. A map $\Lambda: \mbox{ Ell}(A)\rightarrow \mbox{
Ell}(B)$ is said to be a homomorphism if $\Lambda$ gives an order
homomorphism $\lambda_0:K_0(A)\rightarrow K_0(B)$ such that
$\lambda_0:K_0(A)\rightarrow K_0(B)$ such that
$\lambda_0([1_A])=[1_B],$ a homomorphism
$\lambda_1:K_1(A)\rightarrow K_1(B),$ a continuous affine map
$\lambda'_{\rho}:T(B)\rightarrow T(A)$ such that

$$\lambda'_{\rho}(\tau)(p)=r_B(\tau)(\lambda_0([p]))$$

for all projection in $A\otimes M_k$ for $k=1,2,\dots$ and  all
$\tau\in T(B).$

 We say that such $\Lambda$ is an isomorphism, if $\lambda_0$ and
 $\lambda_1$  are isomorphisms and $\lambda'_{\rho}$ is an affine
 homeomorphism. In this case, there is an affine homeomorphism
 $\lambda_{\rho}: T(A)\rightarrow T(B)$ such that $\lambda^{-1}_{\rho}=\lambda'_{\rho}.$
\end{De}

\begin{Theorem}(Corollary 11.9 of \cite{H.Lin8}) Let $A,B\in
\mathcal{A}.$ Then $$A\otimes \mathcal{Z}\cong B\otimes
\mathcal{Z}$$ if
$$Ell(A\otimes \mathcal{Z})\cong Ell(B\otimes
\mathcal{Z}).$$
\end{Theorem}

\begin{De} Recall that a $C$*-algebra $A$ is said to be
$\mathcal{Z}$-stable if $A\otimes \mathcal{Z}\cong A.$ Denote by
$\mathcal{A}_{\mathcal{Z}}$ the class of $\mathcal{Z}$-stable
$C$*-algebras in $\mathcal{A}.$
\end{De}

\begin{Cor} Let $A$ and $B$ be two unital separable amenable simple
$C$*-algebras in $\mathcal{A}_{\mathcal{Z}}.$ Then $A\cong B$ if and
only if $$Ell(A)\cong Ell(B).$$
\end{Cor}

 \section{Main result}

 We recall the definition of the tracial Rokhlin property.

\begin{De} (\cite{Osaka}) Let $A$ be a unital $C$*-algebra and let
$\alpha\in$Aut$(A)$. We say $\alpha$ has the tracial Rokhlin
property if for every $\varepsilon>0$, every $n\in \mathbb{N}$,
every nonzero positive element $a\in A$, every finite set
$\mathcal{F}\subset A,$ there are the mutually orthogonal
projections $e_{0},e_{1},\cdots,e_{n}\in A$ such that:

 (1) $\|\alpha(e_{j})-e_{j+1}\|<\varepsilon$ for $0\leq j \leq
n-1.$

(2) $\|e_{j}b-be_{j}\|<\varepsilon$ for $0\leq j \leq n$ and all
$b\in \mathcal{F}.$

(3) With $e=\sum_{j=0}^{n}e_{j},[1-e]\leq[a].$
\end{De}

The following result is Lemma 1.4 of \cite{Osaka}.

\begin{Lemma} Let $A$ be a stably finite simple unital C*-algebra such that RR(A) = 0
and the order on projections over $A$ is determined by traces. Let $\alpha\in\mbox{Aut}(A).$ Then
$\alpha$ has the tracial Rokhlin property if and only if for every finite set $\mathcal{F} \subset A$, every
$\varepsilon > 0$, and every $n \in \mathbb{N},$ there are mutually orthogonal projections $e_0, e_1, . . . , e_n \in A$
such that:

(1) $\|\alpha(e_{j})-e_{j+1}\|<\varepsilon$ for $0\leq j \leq
n-1.$

(2) $\|e_{j}b-be_{j}\|<\varepsilon$ for $0\leq j \leq n$ and all
$b\in \mathcal{F}.$

(3) With $e=\sum_{j=0}^{n}e_{j}$, we have $\tau(1-e) < \varepsilon$ for all $\tau\in T(A).$
\end{Lemma}

\begin{Theorem} Let $A\in \mathcal{A}$ be a unital separable simple
amenable $C$*-algebra, and let $\alpha$ be an automorphism of $A$.
Suppose that $\alpha$ has the tracial Rokhlin property. Suppose also that
there is an integer $J\geq 1$ such that $[\alpha^J]=[\mbox{id}_A]$ in $KL(A,A).$ Then $A\rtimes_{\alpha}\mathbb{Z}\in \mathcal{A}.$
\end{Theorem}
\begin{proof} Since $A \in \mathcal{A},$ we have $TR(A\otimes
M_{\texttt{P}})=0$ for some supernatural number $\texttt{p}$ of
infinite type. Hence $A\otimes
M_{\texttt{P}}$ have real rank zero, stable rank one and the order on projections over $A\otimes
M_{\texttt{P}}$ is determined by traces(see \cite{H.Lin1}).

For every $\varepsilon>0,$ every $n\in \mathbb{N},$  every finite set
$F\subset A\otimes M_{\texttt{p}}$, Without loss of generality, we
may assume that there exist a finite set $F_A\subset A$ and
$F_{M_{\texttt{p}}}\subset M_{\texttt{p}}$ such that $F=F_A \otimes
F_{M_{\texttt{p}}}.$ We can also find a positive element $a$ of $A$ such that $\tau(a)<\varepsilon$ for all $\tau\in T(A).$

Since $\alpha$ has the tracial Rokhlin property, there exist mutually
orthogonal projections
$$e_0,e_1,\dots,e_{n}\in A$$ such that:

(1) $\|\alpha(e_j)-e_{j+1}\|<\varepsilon$ for $0\leq j\leq n-1.$

(2) $\|e_j b-b e_j\|<\varepsilon$ for $0\leq j\leq n$ and all
$b\in F_A.$

(3) With $e=\sum_{j=0}^{n}e_{j},[1-e]\leq[a].$

 Set $e_j'=e_j\otimes 1_{M_{\texttt{p}}}$ for $ 0\leq j\leq n.$
  Then
$e_j'$, $0\leq j\leq n$,  are mutually
orthogonal projections in $A\otimes M_{\texttt{p}}$.

Then

(1) $\|(\alpha\otimes \mbox{id})(e'_j)-e'_{j+1}\|<\varepsilon$ for $0\leq
j\leq n-1$.

(2) $\|e'_j b-b e'_j\|<\varepsilon$ for $0\leq j\leq n $ and all
$b\in F.$

(3) With $e'=\sum_{j=0}^{n}e'_{j}, \tau(1-e')\leq \tau(a\otimes1_{M_{\texttt{p}}})<\varepsilon$ for all $\tau\in T(A\otimes M_{\texttt{p}}).$

So $\alpha\otimes$id$\in$Aut$(A\otimes M_{\texttt{p}})$ has the tracial
Rokhlin property by Lemma 3.2.

From assumption, there is an integer $J\geq1$ such that
$[\alpha^J]=[\mbox{id}_A]$ in $KL(A,A)$, then $[(\alpha\otimes \mbox{id})^J]=[\mbox{id}_A\otimes \mbox{id}_{M_{\texttt{p}}}]$ in $KL(A\otimes M_{\texttt{p}},A\otimes M_{\texttt{p}}).$

Apply Theorem  4.4 of \cite{H.Lin3}, we get $TR((A\otimes
M_{\texttt{P}})\rtimes_{\alpha\otimes id}\mathbb{Z})=0.$

Next we prove that $(A\otimes M_{\texttt{P}})\rtimes_{\alpha\otimes
id}\mathbb{Z}\cong (A\rtimes_{\alpha}\mathbb{Z})\otimes
M_{\texttt{P}}.$

Let $u$ be the unitary implementing the action of $\alpha$ in the
crossed product $C$*-algebra $A\rtimes_{\alpha}\mathbb{Z}.$

Set of elements which have form
$\sum_{i=-m}^{m}(\sum_{j=0}^{s_i}a_{ij}\otimes m_{ij})(u\otimes
1_{M_{\texttt{P}}})^i$ is dense in $(A\otimes
M_{\texttt{P}})\rtimes_{\alpha\otimes id}\mathbb{Z}.$  We define a
map on dense elements of $(A\otimes M_{\texttt{P}})\rtimes_{\alpha\otimes
id}\mathbb{Z}$  by
$$\varphi: \sum_{i=-m}^{m}(\sum_{j=0}^{s_i}a_{ij}\otimes
m_{ij})(u\otimes 1_{M_{\texttt{P}}})^i\rightarrow
\sum_{i=-m}^{m}(\sum_{j=0}^{s_i}a_{ij}u^i)\otimes m_{ij}.$$

Note
\begin{eqnarray*}&&\sum_{i=-m}^{m}(\sum_{j=0}^{s_i}a_{ij}\otimes
m_{ij})(u\otimes
1_{M_{\mathcal{P}}})^i\\
&=&\sum_{i=-m}^{m}(\sum_{j=0}^{s_i}a_{ij}\otimes
m_{ij})(u^i\otimes
1_{M_{\mathcal{P}}}^i)\\
&=&\sum_{i=-m}^{m}(\sum_{j=0}^{s_i}a_{ij}\otimes m_{ij})(u^i\otimes
1_{M_{\mathcal{P}}})\\&=&\sum_{i=-m}^{m}(\sum_{j=0}^{s_i}a_{ij}u^i)\otimes
m_{ij}. \end{eqnarray*}
then $\varphi$ is a $*$-homomorphism, we can extend $\varphi$ to $(A\otimes
M_{\texttt{P}})\rtimes_{\alpha\otimes id}\mathbb{Z}.$

 So $\varphi$ is a norm invariant
$*$-homomorphism, $\varphi$ is injective, Set of all elements like
the form $\sum_{i=-m}^{m}(\sum_{j=0}^{s_i}a_{ij}u^i)\otimes m_{ij}$
is dense in $(A\rtimes_{\alpha}\mathbb{Z})\otimes M_{\texttt{P}}$.
 We can extend $\varphi$ as an isomorphism  $(A\otimes
M_{\texttt{P}})\rtimes_{\alpha\otimes id}\mathbb{Z}\rightarrow
(A\rtimes_{\alpha}\mathbb{Z})\otimes M_{\texttt{P}}$  .

then $\varphi$ is a $*$-isomorphism.

So $TR((A\rtimes_{\alpha}\mathbb{Z})\otimes M_{\texttt{P}})=0,$
thus $A\rtimes_{\alpha}\mathbb{Z}\in \mathcal{A}.$
\end{proof}

 Next we recall the definition of the  Rokhlin property.

\begin{De}(\cite{A.Kishimoto3})  Let $A$ be a unital simple $C$*-algebra and
$\alpha\in Aut(A).$ We say $\alpha$ has the Rokhlin property if for
every $\varepsilon>0,$ every $n\in \mathbb{N},$ every finite set
$F\subset A,$ there exist the mutually orthogonal projections
$$e_0,e_1,\dots,e_{n-1},f_0,f_1,\dots,f_n\in A$$ such that:

(1) $\|\alpha(e_j)-e_{j+1}\|<\varepsilon$ for $0\leq j\leq n-2$ and
 $\|\alpha(f_j)-f_{j+1}\|<\varepsilon$ for $0\leq j\leq n-1.$

(2) $\|e_j a-a e_j\|<\varepsilon$ for $0\leq j\leq n-1$ and all
$a\in F,$   $\|f_j a-af_j\|<\varepsilon$ for $0\leq j\leq n$ for all
$a\in F$.

(3) $\sum_{j=0}^{n-1}e_j + \sum_{j=0}^n f_j=1.$
\end{De}

If A is a simple unital $C$*-algebra with real rank zero, stable rank one, and has weakly unperforated $K_0(A),$ the Rokhlin property implies
the tracial Rokhlin property(see \cite{Osaka1}).

\begin{Cor} Let $A\in \mathcal{A}$ be a unital separable simple
amenable $C$*-algebra, and let $\alpha$ be an automorphism of $A$.
Suppose that $\alpha$ has the  Rokhlin property. Suppose also that
there is an integer $J\geq 1$ such that $[\alpha^J]=[\mbox{id}_A]$ in $KL(A,A).$ Then $A\rtimes_{\alpha}\mathbb{Z}\in \mathcal{A}.$
\end{Cor}
\begin{proof} Using the similar proof of Theorem 3.3, we can know that $\alpha\otimes$id has the Rokhlin property.
 Since $TR(A\otimes M_{\texttt{p}})=0,$ by Theorem 1.12 of
\cite{Osaka}, $\alpha\otimes \mbox{id}$ has the tracial Rokhlin property. So we can complete the proof as the same as Theorem 3.3.
\end{proof}

The following definition is a stronger version of the Rokhlin property.

\begin{De} Let $A$ be a unital simple $C$*-algebra and
$\alpha\in Aut(A).$ We say $\alpha$ has the cyclic Rokhlin property
if for every $\varepsilon>0,$ every $n\in \mathbb{N},$ every finite
set $F\subset A,$ there exist the mutually orthogonal projections
$$e_0,e_1,\dots,e_{n-1},f_0,f_1,\dots,f_n\in A$$ such that:

(1) $\|\alpha(e_j)-e_{j+1}\|<\varepsilon$ for $0\leq j\leq n-1$ and
$e_n=e_0.$ $\|\alpha(f_j)-f_{j+1}\|<\varepsilon$ for $0\leq j\leq n$
and $f_{n+1}=f_0$.

(2) $\|e_j a-a e_j\|<\varepsilon$ for $0\leq j\leq n-1$ and for
$a\in F.$ $\|f_j a-af_j\|<\varepsilon$ for $0\leq j\leq n$ and all
$a\in F.$

(3) $\sum_{j=0}^{n-1}e_j + \sum_{j=0}^n f_j=1.$
\end{De}

The only difference between the  Rokhlin property
and the  cyclic Rokhlin property is that in condition (1) we require
that $\|\alpha(e_{n-1})-e_0\|<\varepsilon$ and $\|\alpha(f_{n})-f_0\|<\varepsilon.$

\begin{Cor} Let $A\in \mathcal{A}_{\mathcal{Z}}$ be a unital separable simple
amenable $C$*-algebra, and let $\alpha$ be an automorphism of $A$.
Suppose that $\alpha$ has the cyclic Rokhlin property. Suppose also that
there is an integer $J\geq 1$ such that $[\alpha^J]=[\mbox{id}_A]$ in $KL(A,A).$ Then $A\rtimes_{\alpha}\mathbb{Z}\in
\mathcal{A}_{\mathcal{Z}}.$
\end{Cor}
\begin{proof} By Theorem 4.4 of \cite{w.winter3}, if $A$ is $\mathcal{Z}$-stable, $\alpha$ has the
cyclic Rokhlin property, then $A\rtimes_\alpha \mathbb{Z}$ is also
$\mathcal{Z}$-stable, we can get the result from Corollary 3.5.
\end{proof}

\begin{Remark} If $A$ is a unital separable simple $C$*-algebra with locally finite decomposition rank which satisfies the UCT.
Suppose the projections of $A$ separate traces, since $A\otimes M_{\texttt{P}}$
is approximate divisible, apply Theorem1.4(e) of \cite{Blackadar1}, $A\otimes
M_{\texttt{P}}$ has real rank zero. But $A\otimes M_{\texttt{P}}$
has locally finite decomposition rank and is $\mathcal{Z}$-stable, by Theorem 2.1 of \cite{w.winter2},
we have $TR(A\otimes M_{\texttt{P}})=0.$
This $C$*-algebras class is larger than  class of $C$*-algebras which have tracial rank zero.
\end{Remark}

  {\scshape Department of Mathematics, Tongji University,
Shanghai 200092, P.R.CHINA.}

{\it E-mail address}: huajiajie2006@hotmail.com
\end{document}